\documentclass[a4paper]{article}
\usepackage{RR}
\usepackage{hyperref}
\RRdate{Septembre 2008}

\def\THBio{\TH@add{BIO}{Syst\`emes biologiques}}%

\usepackage{mathrsfs}
\usepackage{amsfonts}
\usepackage{amsmath}
\usepackage{amssymb}
\usepackage{graphicx}

\newtheorem{re}{{\bf Remark}}
\newtheorem{thm}{Theorem}

\newtheorem{cor}{{\bf Corollary}}
\newtheorem{defi}{{\bf Definition}}
\newtheorem{prop}{{\bf Proposition}}
\newtheorem{lemme}{{\bf Lemma}}
\newtheorem{prob}{{\bf Problem}}
\newtheorem{ass}{{\bf Assumption}}
\newtheorem{notat}{{\bf Notation}}
\newenvironment{theorem}{\begin{thm} \rm}{\end{thm}}

\newenvironment{corollary}{\begin{cor} \rm}{\end{cor}}

\newenvironment{proposition}{\begin{prop} \rm}{\end{prop}}
\newenvironment{lemma}{\begin{lemme} \rm}{\end{lemme}}

\newenvironment{proof}{\paragraph{Proof:}}{\endpf\\}
\def\endpf{\hfill$\Box$\medskip}

\newcommand{\beq}{\begin{equation}}
\newcommand{\eeq}{\end{equation}}
\newcommand{\ie}{{\it i.e. }}

\newcommand{\eg}{{\it e.g. }}
\newcommand{\dst}{\displaystyle}
\newcommand{\etc}{{\it etc...}}
\RRauthor{Ludovic Mailleret\thanks{INRA, UR 880, F-06903, Sophia Antipolis,
    France ({\tt ludovic.mailleret@sophia.inra.fr}).}
        \and Fr\'ed\'eric Grognard\thanks{INRIA, COMORE, F-06902, Sophia Antipolis, France ({\tt frederic.grognard@sophia.inria.fr}).}}

\RRtitle{Stabilit\'e globale et optimisation d'un mod\`ele impulsionnel g\'en\'eral de
  lutte biologique}

\RRetitle{Global stability and optimisation of a general impulsive biological control model\thanks{This 
        work was supported by the COLOR action ``LutIns\&co'' funded by INRIA Sophia Antipolis M\'editerran\'ee and by the Project ``IA2L'' funded by the SPE Department of INRA.}}

\RRresume{Un mod\`ele impulsionnel de lutte biologique est
  consid\'er\'e; il est compos\'e
  d'un mod\`ele proie-pr\'edateur g\'en\'eral en \'equations diff\'erentielles
  ordinaires compl\'et\'e par une partie discr\`ete d\'ecrivant l'introduction
  p\'eriodique de pr\'edateurs. Il est d\'emontr\'e qu'il existe une solution
  p\'eriodique invariante qui correspond \`a l'\'eradication des proies et une
  condition qui 
  permet d'assurer sa stabilit\'e globale asymptotique est donn\'ee. Un probl\`eme
  d'optimisation li\'e \`a l'utilisation pr\'eventive de la lutte biologique
  est alors consid\'er\'e. Il est suppos\'e que le budget par unit\'e de
  temps de la lutte biologique (c'est-\`a-dire le nombre de pr\'edateurs
  l\^ach\'es) est fixe et la meilleure r\'epartition de ce budget est recherch\'ee
  en termes de fr\'equence d'introduction. La fonction de co\^ut \`a minimiser
  est 
  le temps pris pour r\'eduire une invasion impr\'evue de ravageurs (les
  proies) 
  jusqu'\`a un niveau inoffensif. L'analyse montre que le probl\`eme
  d'optimisation consid\'er\'e admet un nombre infini de solutions. La prise
  en compte de la robustesse n\'ecessaire de la solution 
  optimale indique que la meilleure strat\'egie 
  consiste \`a utiliser des l\^achers de pr\'edateurs aussi fr\'equents (et
  donc aussi faibles) que possible. }

\RRabstract{
An impulsive model of augmentative biological control consisting of a general continuous predator-prey model in ordinary differential equations augmented by a discrete part describing periodic introductions of predators is considered. It is shown that there exists an invariant periodic solution that corresponds to prey eradication and a condition ensuring its global asymptotic stability is given. An optimisation problem related to the preemptive use of augmentative biological control is then considered. It is assumed that the per time unit budget of biological control (\ie the number of predators to be released) is fixed and the best deployment of this budget is sought after in terms of release frequency. The cost function to be minimised is the time taken to reduce an unforeseen prey (pest) invasion under some harmless level. The analysis shows that the optimisation problem admits a countable infinite number of solutions. An argumentation considering the required robustness of the optimisation result is then conducted and it is shown that the best deployment is to use as frequent (and thus as small) as possible predator introductions.}

\RRmotcle{Dynamique proie-pr\'edateur, mod\`ele impulsionnel, stabilit\'e
  globale, optimisation}

\RRkeyword{ 
Predator-prey dynamics, impulsive model, global stability, optimisation.
}

\RRprojet{Comore}
\RRtheme{\THBio}
\URSophia
\begin{document}
\RRNo{6637}

\makeRR


\section{Introduction}

The biological control method intends to reduce harmful organisms populations and relies on the utilisation of their natural enemies, the biological control agents. It can be used to control \eg vectors in vector-borne diseases like malaria or invasive species (may they be animals or plants), but it is mostly dedicated to the control of insect pests, especially at the agricultural cropping system scale \cite{Murdoch1985}. Three main biological control strategies can be identified:
\begin{itemize}
\item conservation biological control: biocontrol agents are already present in the system and are favoured through various means (habitat management...).
\item classical biological control: a small population of biocontrol agents is introduced once in the system and a long term equilibrium between the pests and themselves is targetted.
\item augmentative biological control: biocontrol agents are repeatedly introduced in order to eradicate the pest population.
\end{itemize}

In this contribution we focus on this last method, which is fairly well modelled by an impulsive system of ordinary differential equations (a  continuous system affected by sudden state modifications at some discrete time moments). The intrinsic dynamics of the two interacting populations (pests and biocontrol agents) are described by the continuous part of the system and the repeated (periodic) biocontrol agents introductions  by the discrete one. Early examples of the use of impulsive models in biology are \cite{Ebenhoh1988,Funasaki1993} in chemostat modelling. Since then, they have received considerable attention in the area of epidemiology \cite{D'onofrio2002a,Hui2004a,Shulgin1998}, cancer treatment  \cite{D'onofrio2004,D'onofrio2005,Lakmeche2000}, populations dynamics \cite{Liu2002,Tang2003} or integrated pest management \cite{Liu2005a,Nundloll2008,Tang2005a,Zhang2007a}, to cite some. Though there are numerous studies on bio-inspired impulsive models, to our knowledge only a few of these is concerned with optimisation: see  \cite{Dong2007,Erdlenbruch2007,Meng2007,Tang2006,Xiao2006}. Here we focus both on a global convergence result for a rather general predator-prey model with periodic impulsive predator introductions and on a problem of optimisation on how to best deploy these introductions in the context of crop protection through biological control. Our results follow and generalise  earlier ones that were obtained in a more restricted context  \cite{Mailleret2006}.

This paper is organised as follows. First, the impulsive predator-prey model of augmentative biological control that will be payed attention to throughout the paper is presented. The hypotheses assumed on the biological part (continuous system) are qualitative so that they can encompass a large number of classical biological/ecological functions. The existence of an invariant  periodic prey-eradication solution is proved and a sufficient condition that ensures its global asymptotic stability is given. Then, we focus on an optimisation problem in the perspective of preemptive/preventative use of  augmentative biological control which particularly suits the protection of high valued agricultural crops like \eg orchards, vegetables, mushrooms or ornamentals   \cite{Williams2001,Fenton2001,Hulshof2003,Jacobson2001a,Jacobson2001,Skirvin2002}. It is assumed that a fixed budget ensuring pest eradication is allowed and that only its deployment (a tradeoff between frequent-small or rare-large releases) may be modified. The cost function to be minimised is the time taken to reduce an unforeseen prey (pest) invasion under some harmless level and the input is the impulsive release frequency. It is first shown that the local optimisation problem has a countable infinite number of solutions and, in a second step using an argumentation based on the needed robustness of our result, it is proved that the locally optimum solution is to choose as frequent (and thus as small) as possible releases. It is also shown that this strategy gives a sub-optimal solution of the global optimisation problem. Moreover, the larger the frequency, the lower the cost function, so that even if one cannot achieve a very large frequency, it is always interesting to choose the largest possible one. A Monte Carlo like numerical simulation illustrates our analytical optimisation result and the article is concluded by some recommendations to biological control practitioners that follow from our analysis.

\section{Model Description and Analysis}

\subsection{A general impulsive biological control model}

In this paper, we  assume that the tri-trophic system (crop - pests - biocontrol agents) may be well represented by its bi-trophic simplification representing the prey (pest) predator (biocontrol agent) interactions. The underlying assumption is that the prey population remains at moderate levels so that the crop is not limiting and thus do not need to be taken into account in the model. We consider the simple case in which no inter specific interactions within the predator population affects the predator prey interaction, though some in the prey population may occur (see Remark$~$1 below). According to classical predator-prey modelling, we get the following two-dimensional model:
\begin{equation}\label{continuous_dyn}
\left\{\begin{array}{l}
\dot{x}=f(x)-g(x)y\\
\dot y=h(x)y-my
\end{array}\right.
\end{equation}
$x$ denoting the preys and $y$ the predators. $f(x)$ denotes the preys' growth velocity, $g(x), h(x)$ and $m$ denotes the predators' functional response, numerical response and mortality rate, respectively. 
Since biological processes are always difficult to model, we only assume general qualitative hypotheses on the functions $f(.),~g(.)$ and $h(.)$:

\vspace*{0.5cm}
\noindent \textsc{Hypotheses 1 (H1):} \textit{ Let $f(.),~g(.)$ and $h(.)$ be locally Lipschitz functions on $\mathbb{R}^+$ such that:
\begin{enumerate}\item[(i)] $f(0)=0$
\item[(ii)] $g(0)=0,~g'(0)>0$ and $\forall x> 0,~g(x)> 0$
\item[(iii)] the function $\frac{\dst f(x)}{\dst g(x)}$ is upper bounded for all positive $x$
\item[(iv)] $h(0)=0$ and $\forall x>0,~h(x)>0$
\end{enumerate}}

{\it Remark 1:} A large part of the predator-prey functions encountered in the literature fit these hypotheses: (H1-i) only indicates that no spontaneous generation of pests (prey) is possible; $f(.)$ might be constant, linear, logistic or model an Allee effect \etc (H1-ii) means that there is no prey consumption as the prey level is zero and that if some prey are present, the predator are able to find and consume them. Moreover, $g(x)$ is supposed to be increasing at the origin. As a consequence $g(x)$ may be a Holling I or II function. $g(x)$ might even be non-monotonic (in case of \eg prey group defence abilities) like in Holling IV, but in this case from (H1-iii) $f(x)$ must be of type IV too or become negative for high $x$ values (like for logistic growth). 
Notice however that Holling III like functional response (\ie those with null derivative at $x=0$) fall beyond the scope of this study. No hypotheses are made on the numerical response $h(x)$ except (H1-iv) which is quite natural. Indeed, as it will be clearer in the following, our argumentation is mainly based on the fact that $g(x)$ is positive for positive $x$.

We now model the augmentative biological control procedure as the $T$-periodic  releases  of biological control agents (with $T>0$). Releases are, by their very nature, discrete phenomena: every time $nT$ ($n\in\mathbb{N}$) some predators are instantly added into the system. Let us suppose that there is a fixed rate of predators $\mu$ (\ie number of predators per unit time) to be released in the system. Notice that such a quantity $\mu$ is a direct measure of the costs of the biological control method over some time period. Hence in the following it will be referred to as well as the release rate or the biological control budget. 

This results in the following: at each time period $T$, $\mu T$ predators are added to the predator population $y$, yielding the discrete process:
\begin{equation}\label{dicrete_releases}\forall n\in\mathbb{N}, ~y(nT^+)=y(nT)+\mu T\end{equation}

Combining \eqref{continuous_dyn} and (\ref{dicrete_releases}) yields the following impulsive biological control model, similar to, but more general than the ones presented \eg in \cite{Liu2005,Liu2005a}:
\begin{equation}\label{model}
\left\{\begin{array}{l}
\dot x=f(x)-g(x)y\\
\dot y=h(x)y-my\\
\forall n\in\mathbb{N},~y(nT^+)=y(nT)+\mu T
\end{array}\right.
 \end{equation}

{\it Remark 2:} It is quite easy to check from (H1) that the proposed model is, as required for biological models in general, a non-negative system (see \eg \cite{Mailleret2008}), \ie the non-negative orthant of the state space is positively invariant by system \eqref{model}. Then every non-negative initial conditions $x_0$ and $y_0$ at  initial time $t_0$ (that only make sense with respect to the considered problem) produce through model \eqref{model} non-negative trajectories $x(t)$ and $y(t)$ for all positive time.

\subsection{Global convergence}

In this section we show that provided the predator release rate $\mu$ is greater than some value, augmentative biological control is able to drive any pest population to zero. This result is summarised into the following Theorem.

\begin{theorem}\label{theo_stabi}
Under Hypotheses 1, model (\ref{model}) possesses a ``pest free" $T$-periodic solution:
\begin{equation}\label{sol_proph}
\left(x_p(t),y_p(t)\right)=\left(0,\frac{\dst\mu Te^{-m(t\mod{T})}}{1-e^{-mT}}\right)
\end{equation}
which is locally asymptotically stable if and only if:
\beq\label{cond_stab} 
\mu>\frac{m f'(0)}{g'(0)}
\eeq 
and globally asymptotically stable if:
\beq\label{cond_stab_glob} 
\mu>S\triangleq \sup_{x\geq 0} \frac{mf(x)}{g(x)}
\eeq
\end{theorem}

{\it Remark 3:} Note that from (H1-iii) $S$ as defined in \eqref{cond_stab_glob} does exist.

\begin{proof} 
We first focus on the ``pest free'' set: $\{(x,y)\in\mathbb{R}^2_+,~x=0\}$, which is clearly invariant by system (\ref{model}) through (H1-i) and (H1-ii). Within this set, according to (H1-iv), system (\ref{model}) becomes:
\begin{equation}\label{model_prophy}
\left\{\begin{array}{l}
\dot x=0\\
\dot y=-my\\
\forall n\in\mathbb{N},~y(nT^+)=y(nT)+\mu T
\end{array}\right.
 \end{equation}
that yields: $y((n+1)T^+)=y(nT^+)e^{-mT}+\mu T$. It is clear that the sequence $(y(nT^+))_{n\in\mathbb{N}}$ has a single and globally stable equilibrium $y^\star=\mu T/(1-e^{-mT})$. Then, system (\ref{model_prophy}) trajectories globally converges towards the $T$-periodic solution:
\[\forall t\in(nT,(n+1)T],~(x(t),y(t))=\left(0,\frac{\mu T e^{-m(t-nT)}}{1-e^{-mT}}\right)
\]
which is indeed the same as (\ref{sol_proph}), so that the existence of the $T$-periodic pest free solution for system \eqref{model} has been proved.

We now concentrate on the stability of the pest free solution (\ref{sol_proph}) for system (\ref{model}), that is to say we do not restrict ourselves to the pest free set. We first change system \eqref{model} variables to consider the  deviations from the pest free solution that are denoted $(\tilde x,~\tilde y)$ so that:
\[(\tilde{x}(t),\tilde{y}(t))=(x(t),y(t))-(x_p(t),y_p(t))
\]
that yields:
\beq\label{sys_eps}
\left\{\begin{array}{l}
\dot{\tilde{x}}=f(\tilde{x})-g(\tilde{x})(\tilde{y}+y_p(t))\\
\dot{\tilde{y}}=h(\tilde{x})(y_p(t)+\tilde{y})-m\tilde{y}
\end{array}\right.
\eeq

We first investigate local asymptotic stability (LAS) of the periodic solution $(x_p(t),~y_p(t))$. Assuming $(\tilde x,~\tilde y)$ are small, we get from \eqref{sys_eps} at first order in $\tilde{x}$ and $\tilde{y}$:
\begin{equation}\label{sys_lin}
\left\{\begin{array}{l}
\dot{\tilde{x}}=(f'(0)-g'(0) y_p(t))\tilde{x}\\
\dot{\tilde{y}}=h'(0)y_p(t) \tilde{x}-m \tilde{y}
\end{array}\right.
\end{equation}
which is a linear system (in $\tilde{x}$ and $\tilde{y}$) with $T$-periodic
coefficients. The LAS of $(x_p(t),y_p(t))$ for the original system \eqref{model} is equivalent to the LAS of $(0,0)$ for system \eqref{sys_lin}. The classical approach (Floquet Theory) is to compute the discrete dynamical system that maps the variables at some time to the variables one period of the periodic coefficients later. The origin is then asymptotically stable if and only if the obtained discrete dynamical system is asymptotically stable (see \eg \cite{Chicone1999} for the theory and \cite{Liu2005a} for an application comparable to ours). Here, we get:
\[\left(\begin{array}{c}
\tilde x\\
\tilde y
\end{array}
\right)\left((n+1)T^+\right)=\left(\begin{array}{cc}
\dst e^{\int_{nT}^{(n+1)T} f'(0)-g'(0)y_p(\tau)d\tau} & 0\\
\dag & \dst e^{-mT}                         \end{array}\right)
\left(\begin{array}{c}
\tilde x\\
\tilde y
\end{array}
\right)\left(nT^+\right)
\]
the matrix being lower triangular, the $\dag$ term does not influence the system's stability so there is no need to compute it. Clearly $e^{-mT}$   belongs to $(0,1)$, and \linebreak$e^{\int_0^T(f'(0)-g'(0)y_p(\tau))d\tau}$ is positive. Then, $(0,0)$ is LAS for \eqref{sys_lin} (\ie $(x_p,y_p)$ is LAS for \eqref{model}) iff this latter coefficient is lower than one \ie:
\begin{equation}\begin{array}{ccl}\label{LAS}
\dst \int_0^T(f'(0)-g'(0)y_p(\tau))d\tau<0 &~\Leftrightarrow~& f'(0)T<\dst\int_0^Tg'(0)\frac{\dst\mu Te^{-m\tau}}{1-e^{-mT}}\ d\tau\\[.35cm]
&~\Leftrightarrow~&\mu>\dst\frac{mf'(0)}{g'(0)}
  \end{array}
\end{equation}
so that the LAS of $(x_p,y_p)$ part of the Theorem is shown.

We now focus on the global asymptotic stability (GAS) of the pest free solution. From now on, we assume that system \eqref{model} is initiated at $(x_0,y_0)$ at time $t_0\geq 0$, \ie system \eqref{sys_eps} is initiated at $(\tilde{x_0},\tilde{y_0})=(x_0,y_0-y_p(t_0))$ at time $t_0\geq 0$.\\ Let us consider the following function:
\[G(\tilde{x})=m\int_{\varkappa}^{\tilde{x}} \frac{1}{g(s)}\ ds\]
for some $\varkappa>0$. Since from (H1-ii) $g(x)$ is positive for positive $x$, $G(.)$ is an increasing function from $\tilde{x}=0$, where it goes to $-\infty$ since $g(.)$ is locally Lipschitz on $\mathbb{R}^+$. In the following we  investigate $G(\tilde{x}(t))$ behaviour as time $t$ goes to infinity. We have:
\begin{eqnarray}
G(\tilde{x}(t))-G(\tilde{x}_0)&=& m\dst \int_{\tilde{x}_0}^{\tilde{x}(t)} \frac{1}{g(s)}\ ds \nonumber \\
&=& m\dst \int_{t_0}^{t} \frac{\dot{\tilde{x}}(\tau)}{g(\tilde{x}(\tau))}\ d\tau \nonumber\\
&=& \dst \int_{t_0}^{t}\left( \frac{mf(\tilde{x}(\tau))}{g(\tilde{x}(\tau))}-m(\tilde{y}(\tau)+y_p(\tau))\right)\ d\tau\label{G1}
\end{eqnarray}
Considering system \eqref{sys_eps} together with (H1-iii) and Remark 2, we get:
\[\dot{\tilde{y}} \geq -m \tilde{y}
\]
so that, through standard arguments:
\beq\label{y_minim}\tilde y(t)\geq \min(0,\tilde y_0) e^{-m(t-t_0)}\eeq
Using \eqref{G1} and the definition of $S$ in \eqref{cond_stab_glob}, we then have for all $t\geq t_0$:
\begin{eqnarray}
G(\tilde{x}(t))-G(\tilde{x}_0)&\leq& \dst \int_{t_0}^{t} \left(S-m y_p(\tau)-m\min(0,\tilde y_0) e^{-m(\tau-t_0)}\right) d\tau\nonumber \\
&=& \int_{t_0}^{\left(\left\lfloor\frac{t_0}{T}\right\rfloor+1\right)T} \left(S-m y_p(\tau)\right)d\tau\nonumber\\&&+\left(\left\lfloor\frac{t}{T}\right\rfloor-\left\lfloor\frac{t_0}{T}\right\rfloor-1\right)\int_0^T\left(S-m y_p(\tau)\right)d\tau\nonumber\\
&& +\int_{\left\lfloor\frac{t}{T}\right\rfloor T}^t \left(S-m y_p(\tau)\right)d\tau+\min(0,\tilde y_0) \left(e^{-m(t-t_0)}-1\right)~~~~~~~~~~\label{G2}
\end{eqnarray}
since $y_p(t)$ is $T$-periodic. It is clear that the first, third and fourth term of the right hand side of \eqref{G2} are upper bounded. Now, suppose that \eqref{cond_stab_glob} holds, then in a similar way than in \eqref{LAS}:
\[\int_0^T\left(S-m y_p(\tau)\right)d\tau<0
\]
so that, since $\left\lfloor\frac{t}{T}\right\rfloor$ goes to infinity as $t$ does, the right hand side of \eqref{G2} goes to $-\infty$. Then, condition \eqref{cond_stab_glob} implies:
\[\lim_{t\rightarrow+\infty} G(\tilde{x}(t))=-\infty~~\Leftrightarrow~~\lim_{t\rightarrow+\infty} \tilde{x}(t)=0
\]
One can now prove that $\tilde{y}$ converges to zero as well. Indeed, from \eqref{y_minim}:
\[\tilde{y}_0\leq 0~~\Rightarrow~~\tilde y(t)\geq \tilde{y}_0e^{-m(t-t_0)}\]
so that, either $\tilde{y}(t)$ converges to zero in infinite time from below, or it reaches the positively invariant region where $\tilde{y}\geq 0$ in finite time. Up to initial time $t_0$ translation, we then have to consider the positive $\tilde{y}_0$ only. Since $\tilde{x}$ converges to zero and $h(0)=0$, it is clear that there exists a time $t_f$ such that:
\[\forall t> t_f,~h(\tilde{x}(t))\leq \frac{m}{2}~~\Rightarrow~~\forall t> t_f,~\dot{\tilde{y}}\leq h(\tilde{x})y_p(t)-\frac{m}{2}\tilde{y}\]
Since $h(\tilde{x})y_p(t)$ goes to zero as $t$ goes to infinity, so does $\tilde{y}$. We have shown that if condition \eqref{cond_stab_glob} holds, $(0,0)$ is globally attractive for \eqref{sys_eps}, \ie $(x_p(t),y_p(t))$ is globally attractive for \eqref{model}.\\
It is then easily checked through the Hospital rule, (H1-i) and (H1-ii) that:
\[\lim_{x\rightarrow 0^+} \frac{f(x)}{g(x)}=\frac{f'(0)}{g'(0)}
\] 
Then, condition \eqref{cond_stab_glob} implies that the necessary and sufficient condition \eqref{cond_stab} for LAS of $(x_p,y_p)$ holds true, which concludes the proof.\end{proof}

Theorem \ref{theo_stabi} ensures that when the condition \eqref{cond_stab_glob} is verified, the extinction of the prey (pest) population is GAS on the positive orthant of the state space. If the local  condition \eqref{cond_stab} were not satisfied, the pest free solution would then no more be stable. Using bifurcation theory for impulsive systems proposed by \cite{Lakmeche2000}, one then should be able to perform a bifurcation analysis, similar to what is done \eg in \cite{Liu2005a,Negi2007} to show the rich dynamics that system \eqref{model} would produce. However, as stated in the introduction, with respect to our application we concentrate here on the pest eradication problem only. In the case where condition \eqref{cond_stab} holds true, but condition \eqref{cond_stab_glob} does not, the pest free solution is still locally stable but, since our global stability condition \eqref{cond_stab_glob} is only sufficient, we cannot rule out the possibility that it is also globally stable. Such a release rate has the advantage of being smaller than the one required for GAS so that it would be cheaper for the grower to choose a one as such: it allows good control of limited pests invasions. The price to pay is however the risk of considerable crop damages if a large pest invasion  occurs.

It is to be noted that the stability conditions \eqref{cond_stab} and \eqref{cond_stab_glob} on the release rate $\mu$ are both independent of the release period $T$. Therefore, once a release rate $\mu$ fulfilling a  stability condition is chosen, pest eradication could be achieved independently of the choice of the release period $T$. Indeed, both infrequent and large releases ($T$ and $\mu T$ large) or frequent and small  releases ($T$ and $\mu T$ small) would ultimately drive the pest population to zero. Though there is no difference in the long term, different values of the release period $T$ might however result in different outcomes for the transient dynamics of the system. Hence in the following, we will take advantage of the adjustable parameter $T$ to address the practically important question: how to ``best" (in a sense that will be detailed in the sequel) deploy  a given biological control budget $\mu$ to eradicate pest outbreaks ?

\section{Optimal Choice of the Release Period in the Preemptive Case\label{sec:opt}}

\subsection{Statement of the optimisation problem}\label{problem}

Throughout the following, we assume that the condition for stability  of the pest free solution holds true. We focus on preemptive/preventative use of biological control agents releases: we suppose that biological control agents are released in anticipation of pests outbreaks, so that natural enemies would be able to fight the pest right at the time of their invasion. This departs from most classical biological control procedures  which are of a more ``feedback" type: biocontrol agents are released once the pests have been detected only. Such a preventative approach has however been payed attention to since it appears as achieving more acceptable pest control, especially for high valued crops that are very sensitive to the slightest pest outbreak (see \cite{Fenton2001,Skirvin2002,Fenton2002}  for theoretical/simulatory studies and \cite{Ehler1997,Jacobson2001,Jacobson2001a,Williams2001,Hulshof2003} for real life experiments and a biological perspective). Hence in the following, we suppose that preventative releases are being performed and that the system is in the invariant set $\{\tilde y\geq 0\}$ as a pest population $x_0$ invades the crop at some unforeseen moment $t_0$.

Our goal is to provide recommendations on how to deploy (\ie how to choose $T$) such a preemptive biological control strategy to minimise crop damages due to a pest outbreak. To evaluate crop damages induced by the pests, we use the the concept of ``Economic Injury Level" (EIL) that has been introduced from the early bases of theoretical biological control \cite{Stern1959}. EIL (denoted $\bar{x}$ in the sequel) is defined as the ``lowest (positive) pest population level that will cause economic losses on the crop". $\bar{x}$ is then a constant parameter of our problem. For a given pest outbreak $x_0>\bar{x}$, we consider that the damage cost $J$ is an increasing function of the time spent by the pest population above the EIL $\bar{x}$, \ie:
\[J=\int_{t_0}^{t_f} \gamma(\tau)d\tau\]
with $t_f$ defined such that $x(t_f)=\bar{x}$ and $\gamma(.)>0$. It is clear that minimising $J$ is equivalent to minimising $\Pi=(t_f-t_0)$ so that we will only concentrate on this latter in the following. 

We are left with two more problems. The first one is that, for given $x_0$ and $T$, the damage cost $J$ (or $\Pi$) strongly depends on the pest outbreak instant $t_0$, \ie for a given $T$, the same invading pest level $x_0$ yields different damage cost depending on the value of $t_0$. Without loss of generality, from now on we will consider that $t_0\in[0,T)$. We seek to provide the safest biological control procedure so that we will concentrate on the damage cost for the worst $t_0$, ensuring any other $t_0$ would result in smaller damage costs. Hence in the following, we seek the $T$ that  minimises the damage cost for its worst $t_0$ \ie we seek to minimise $\dst\max_{t_0\in[0,T)}\Pi$.

The remaining problem is that our analysis cannot be achieved for the general model \eqref{model}. It is however possible to carry it out for a local approximation of the $\dot x$ part of the model about $x=0$ as well as for an upper-bound of $\dot{x}$. Indeed, consider the local approximation of $\dot x$ near $(x_p,y_p)$. From \eqref{sys_lin}, and considering the change of variable $z_1=\frac{m}{g'(0)}\ln\left(\frac{x}{\bar x}\right)$ (which is obviously increasing in $x$), we get:
\[\dot{z}_1=S_l-m y_p(t)\]
where $S_l=\frac{mf'(0)}{g'(0)}\leq S$.

Now we consider the general $\dot x$ equation of model \eqref{model} under the hypothesis that the system is in the positively invariant set  $\{\tilde{y}\geq 0\}$. We get for all $t\geq t_0$:
\[\dot x\leq f(x)-g(x)y_p(t)
\]
Let us make the change of variable $z_2=m\int_{\bar{x}}^x \frac{1}{g(x)}dx$ (which is increasing in $x$), we get:
\[\dot z_2\leq S-my_p(t)
\]
whose right hand side is similar to the $\dot z_1$ equation. 

In the sequel we will then focus on the system:
\begin{equation}\label{zsystem}
\dot z=\sigma-my_p(t)
\end{equation}
with $z$ standing for $z_1$ or $z_2$ and $\sigma$ for $S_l$ or the original $S$. Through equation \eqref{zsystem} we will thus study in one step the exact local approximation of our system as well as an upper bound of it, yielding a locally optimal result as well as a globally sup-optimal one.

\subsection{Preliminary results}

To state our main optimisation result, we first need some preliminary computations. Since we assume that the stability condition of the pest free solution holds true, we have:

\vspace*{0.5cm}
\noindent \textsc{Hypotheses 2 (H2):} \textit{The biological control budget is chosen such that: $\mu>\sigma$}

We first show that provided the release period $T$ is not too large, the pest population is decreasing for $t\geq t_0$. 

\begin{proposition}\label{prop:Tm}
There exists a release period $\hat{T}$ such that if $T<\hat{T}$, the pest population is decreasing for $t\geq t_0$.
\end{proposition}

\begin{proof} Consider equation \eqref{zsystem}. Provided   $\min_t{y_p(t)}>\frac{\sigma}{m}$, $z$ is decreasing for $t\geq t_0$; then so does the pest population $x$. Notice that:
\[\min_t y_p(t) =\frac{\mu T}{e^{mT}-1}\]
is a decreasing function of $T$ with:
\[\lim_{T\rightarrow 0^+}\min_t y_p(t)=\frac{\mu}{m}>\frac{\sigma}{m}\] 
since the stability condition holds. 

Then Proposition \ref{prop:Tm} holds true with $\hat{T}$ solution of:
\beq\label{tm}\frac{\mu \hat{T}}{e^{m\hat{T}}-1}=\frac{\sigma}{m}
\eeq\end{proof}

Proposition \ref{prop:Tm} has practical interest regarding growers concerns: indeed it is not interesting to invest in preventative releases of biological control agents and allow an invading pest population to proliferate after the moment of invasion. It is much more desirable to choose a procedure that will always make the pest population decreasing after the invasion. Thus in the sequel we assume that:

\newpage
\vspace*{0.5cm}
\noindent \textsc{Hypotheses 3 (H3):} \textit{The release period is chosen such that $T\in(0,\hat{T})$.}

We now seek the moment of invasion $t_0$ that maximises the damage costs $J$ (\ie  that maximises the damage time $\Pi$) for a given invading pest population $x_0>\bar x$ (corresponding to $z_0>0$) and a given release period $T\in(0,\hat{T})$.  Notice that with the $z$ formalism, we have $\Pi=t_f-t_0$ with $z(t_0)=z_0$ and $z(t_f)=0$. We have:

\begin{lemma}[\cite{Mailleret2006}]\label{lem:worst_t0}
Suppose $z_0>0$ and $T\in(0,\hat{T})$ are fixed, then one of the following holds:
\begin{enumerate}\item[(i)] $\exists k\in\mathbb{N},\ (\mu-\sigma) kT=z_0$\ then\ $\Pi=kT$\vspace*{.05cm}
\item[(ii)] $\exists k\in\mathbb{N},~\max_{t_0}\Pi(t_0)=\Pi(t_0^*)=(k+1)T-t_0^*$ (\ie  $z((k+1)T)=0$)\end{enumerate}\end{lemma}

We will first show case (i) and prove that, otherwise, either  case (ii) holds or $\Pi$ is maximum at $t_0^*=0$, this latter case being impossible.

\begin{proof}
Suppose that $\Pi(t_0)$ is maximum for $t_0^*\in (0,T)$ and such that $(t_0^*+\Pi(t_0^*)) \in (kT, (k+1)T)$ for some integer $k$. That is to say that we suppose that $\Pi(t_0)$ is maximum for both $t_0$ and $t_f=(\Pi(t_0)+t_0)$ being strictly within two moments that are multiples of $T$.

Now pick $t_0^m<t_0^*<t_0^M$ in $(0,T)$ such that  $(t_0+\Pi(t_0)) \in [kT, (k+1)T)$ for all $t_0\in[t_0^m,t_0^M]$. Within this set, define  $\theta(t_0)$ as:
\beq\label{def_theta}\Pi(t_0)=kT -t_0+\theta(t_0)\eeq
Integrating (\ref{zsystem}) between $t_0$ and $t_0+\Pi(t_0)$, we get:
\begin{eqnarray}\label{to_deriv_theta}
z(t_0+\Pi(t_0))&=&z_0+\sigma\Pi(t_0)-m\dst \int_{t_0}^{t_0+\Pi(t_0)} y_p(\tau)d\tau\nonumber\\
&=&z_0+\sigma\Pi(t_0)\nonumber\\&&-\dst\frac{m\mu T}{1-e^{-mT}}\dst\left(\int_{t_0}^T e^{-m\tau}d\tau +(k-1)\int_0^T e^{-m\tau}d\tau+\int_0^{\theta(t_0)}e^{-m\tau}d\tau
\right)\nonumber\\
&=&z_0+\sigma\Pi(t_0)-\dst\frac{\mu T}{1-e^{-mT}} \left(e^{-mt_0}-e^{-m\theta(t_0)}+k(1-e^{-mT} )\right)
\end{eqnarray}
which is, from the definition of $\Pi$, equal to $0$. 

From \eqref{def_theta}, we have:
\beq\label{eq:deriv_pi}\frac{d\Pi}{dt_0}(t_0)=\frac{d\theta}{dt_0}(t_0)-1
\eeq
Using this and differentiating (\ref{to_deriv_theta}) with respect to $t_0$ yields:
\beq\label{deriv_theta}
\frac{d \theta}{d t_0}(t_0)=\dst\frac{ \left(\frac{m\mu T}{1-e^{-mT}}\right)e^{-mt_0}-\sigma}{ \left(\frac{m\mu T}{1-e^{-mT}}\right)e^{-m\theta(t_0)}-\sigma}
\eeq
To have a maximum of $\Pi$ at $t_0^*$, it is required that:
\[\frac{d \Pi}{d t_0}(t_0^*)=0~~ \Leftrightarrow ~~\frac{d \theta}{d t_0}(t_0^*)=1\] 
so that from \eqref{deriv_theta}, we must have $\theta(t_0^*)=t_0^*$. Then, from (\ref{to_deriv_theta}) and \eqref{def_theta}, the integer $k$ must be  such that:
\[(\mu-\sigma)kT=z_0\]
and $\Pi=kT$ does not depend on $t_0$, which show case (i). Otherwise $\Pi(t_0)$ has no extremum at $t_0^*$ verifying $t_0^*\in (0,T)$ and $(t_0^*+\Pi(t_0^*))\in(kT,(k+1)T)$.

Two cases remain  to be studied: either $\Pi$ is maximum for $t_0=0$, or $t_0^*$ is such that $(t_0^*+\Pi(t_0^*))=(k+1)T$. The remaining cases $t_0^*=T$ and $(t_0^*+\Pi(t_0^*))=kT$ might be studied by $k$ reparametrisation. 

We concentrate first on the case $t_0^*=0$. Then from (\ref{eq:deriv_pi}) and (\ref{deriv_theta}) the right derivative of $\Pi$ at $t_0=0$ is:
\beq\label{eq:dpi_0}\frac{d \Pi}{d t_0}(t_0=0^+)= \frac{ \left(\frac{m\mu T}{1-e^{-mT}}\right)-\sigma}{ \left(\frac{m\mu T}{1-e^{-mT}}\right)e^{-m\theta(t_0)}-\sigma}-1\eeq
From Proposition \ref{prop:Tm}, we have $\forall T\in(0,\hat{T})$:
\[\sigma<\frac{m\mu T}{(1-e^{-mT})}e^{-mT}
\]
Since by definition $\theta(t_0)\leq T$ both the numerator and denominator of the fraction in the right hand side of \eqref{eq:dpi_0} are positive, the former being larger than the latter. $\frac{d \Pi}{d t_0}(t_0=0^+)$ is then positive so that $\Pi$ has a (local) minimum at $t_0=0$, ruling out the case $t_0^*=0$.

We now focus on the case where $t_0^*$ is such that $t_0^*+\Pi(t_0^*)=(k+1)T$. Then from (\ref{def_theta}), $\theta(t_0^*)=T$ and the left derivative of $\Pi$ at $t_0^*$ is:
\[\frac{d \Pi}{d t_0}(t_0^{*-})=\frac{ \left(\frac{m\mu T}{1-e^{-mT}}\right)e^{-mt_0^*}-\sigma}{ \left(\frac{m\mu T}{1-e^{-mT}}\right)e^{-mT}-\sigma}-1\]
In a very same way as in the previous case, since $T<\hat{T}$ we show that this derivative is positive. Similarly one can show that the right derivative of $\Pi$ at $t_0^*$ is negative. This can be performed through the reparametrisation of $k$ as $(k+1)$ while noticing that this corresponds to $\theta(t_0^*)=0$. 

Then for the considered $z_0$ and $T$, $\Pi$ is maximum for $t_0^*$ such that $(t_0^*+\Pi(t_0^*))=(k+1)T$, which completes the proof of case (ii).\end{proof}

Lemma \ref{lem:worst_t0} is quite natural. Indeed, from the linearity and periodicity of \eqref{zsystem}, the decrease that takes place  between times $T$ and $kT$ is independent of $t_0$. The worst case should then contain the end of the first time interval where the predators are scarce, rather than the beginning of the last interval where they are abundant. 

In the following we look for the release period $T\in(0,\hat{T})$ that minimises the worst case damage time: $\max_{t_0}\Pi$.

\subsection{Optimal Choice of the Release Period}

We now investigate which $T \in(0,\hat{T})$ minimises the worst case damage time $\max_{t_0}\Pi$ (\ie that minimises the worst case damages $\max_{t_0}J$). We get the following result:
\begin{theorem}\label{best_period}
Suppose $z_0>0$ is fixed. Let:
\beq\label{t1def}T_1=\frac{z_0}{\mu-\sigma}\eeq
Then there exists $n_0\in\mathbb{N}^*$ such that
$\dst\min_{T\in (0,\hat{T})}\max_{t_0}\Pi=T_1$ is reached at $T=T_n=\dst \frac{T_1}{n}$ for all integer $n>n_0$.
\end{theorem}

\begin{proof}
Consider $T_1$ that is positive since $\mu$ is assumed to be greater than $\sigma$. It is then clear that there exists an integer $n_0$ such that for all integer $n> n_0, T_n<\hat{T}$. Now consider a $T_n$ with $n>n_0$; using (\ref{t1def}), we have:
\[(\mu-\sigma)nT_n=(\mu-\sigma)T_1=z_0\]
Such a $T_n$ corresponds to case (i) in Lemma \ref{lem:worst_t0}, thus we have $\Pi=T_1$ which does not depend on $t_0$. 

Now we show that $\Pi=T_1$ is the minimum (with respect to $T\in(0,\hat{T})$) of $\max_{t_0}\Pi$. Pick a $T\neq T_n$ for all $n>n_0$. For such a $T$, case (i) of Lemma \ref{lem:worst_t0} is ruled out and according to case (ii), we have:
\[\exists k\in\mathbb{N},~\max_{t_0}\Pi=\Pi(t_0^*)=(k+1)T-t_0^*\]
and $z((k+1)T)=0$. Using (\ref{to_deriv_theta}) and $\theta(t_0^*)=T$, we have:
\beq \label{eq:t0z0} z_0+\sigma\left((k+1)T-t_0^*\right)-(k+1)\mu T+\mu T\frac{1-e^{-mt_0^*}}{1-e^{-mT}}=0\eeq
so that:
\[(k+1)T=\frac{1}{\mu-\sigma}\left(
z_0-\sigma\ t_0^*+\mu T\frac{1-e^{-mt_0^*}}{1-e^{-mT}}
\right)
\]
which yields:
\begin{eqnarray}
 \max_{t_0}\Pi&=&(k+1)T-t_0^* \nonumber\\
&=& T_1+
\frac{\mu}{\mu -\sigma}\left(\frac{1-e^{-mt_0^*}}{1-e^{-mT}}\ T-t_0^*
\right)\label{end_comes}\end{eqnarray}

Notice that the function $\left(\frac{1-e^{-mt}}{t}\right)$ is decreasing of $t$, so that, since $t_0^*<T$, we have:
\[\left(\frac{1-e^{-mt_0^*}}{1-e^{-mT}}\ T-t_0^*
\right)> 0\]
Since from (H2), $\mu>\sigma$, we have shown that $\max_{t_0}\Pi$ is greater than $T_1$ for all period $T$ such that $\forall n>n_0,~T\neq T_n$, which completes the proof.\end{proof}

Theorem \ref{best_period} shows that for a given $z_0$ we have a countable infinite number of release periods, the $T_n$, that solves our optimisation problem. 

However, we have two difficulties with Theorem \ref{best_period}. On the one hand, through equation \eqref{zsystem} we studied both the local optimisation problem and the global suboptimal one. To solve these two, we need to choose a release period $T=T_n$, which depends on $\sigma$. In the local optimisation case, $\sigma$ equals $S_l=\frac{mf'(0)}{g'(0)}$ while in the global suboptimal one, $\sigma$ equals $S=\sup_{x\geq 0}\frac{m f(x)}{g(x)}$. There is no reason why $\frac{S}{S_l}$ would be rational, so that Theorem \ref{best_period} states that it is not possible to choose a release period $T$ that solves both the local and global optimisation problem (what remains however the objective).

On the other hand, it is to be noticed that $T_1$, which determines the optimal values of the release period, depends on some ``biological" values, namely the invading pest population level $x_0$ (\ie $z_0$) and the combination of parameters $\sigma$. This is quite a big drawback since biological parameters are usually badly known and the initial invading pest level $x_0$ is not known at all. Clearly such model-based optimisation approaches should take into account these uncertainties \cite{Loehle2006}. 

Hence, in the following, we will show how $T$ should be
chosen in order to ensure that the worst case damage time $\max_{t_0}\Pi$ remains close to its minimal value, independently of the value of $x_0$ and $\sigma$. In doing so, we will provide an almost optimal choice of $T$ that solves both the local and global optimisation problem and that is robust to the value of $x_0$ and to parameters uncertainty.

\subsection{Robustness Analysis}

In the preceding sections, we  wanted our choice of $T$ to minimise the difference between the time that it takes to reach $\bar x$ from $x_0$ for the worst possible choice of the initial time $t_0$. In addition to that, we now also want this to be robust with respect to the uncertainty on the initial condition $x_0$ and to the model parameters that we will denote $p$ in the sequel ($p$ is simply the pair $(\sigma, m)$. Indeed, in the
poorly measured conditions of crop culture and with the accompanying scarcely
known models, it is fundamental for a control method to be robust with respect
to these uncertainties. We then suppose that the ``true" $x_0$ and $p$ belongs to some compact sets:

\vspace*{0.5cm}
\noindent \textsc{Hypotheses 4 (H4):} {\it The uncertainties on the initial condition and the model parameters are bounded, with $z_0\in[\underline{z_0},\overline{z_0}]$ denoted $\mathcal{Z}$ and 
$p$ that belongs to a compact set $\mathcal{P}$.}

We of course still assume that the invading population is greater than the EIL $\bar{x}$ (\ie $\underline{z_0}>0$) and that the stability condition holds, \ie $\mu>\max_{p\in\mathcal{P}}(\sigma)$. Our optimisation procedure will be robust provided our choice of $T$ minimises $\max_{t_0}\Pi$ for the worst $z_0$ and $p$ choice within the sets defined in Hypotheses (H4). 

Notice that the minimum (with respect to $T$) of $\max_{t_0} \Pi$ is equal to $T_1$ that does depend on both $z_0$ and $p$. Thus in our robustness analysis it makes more sense to consider the deviation of $\max_{t_0} \Pi$ from $T_1$,  rather than the absolute value of $\max_{t_0} \Pi$. Hence, we first focus on the evolution of:
\[\max_{(z_0,p)\,\in\,\mathcal{Z}\times\mathcal{P}}\left[\left(\max_{t_0} \Pi(t_0,z_0,p)\right)-T_1(z_0,p)\right]\]
with respect to $T$. We have:

\begin{theorem} \label{theo:rob}
The function:
\beq\label{rob:deviation}T\mapsto \max_{(z_0,p)\,\in\,\mathcal{Z}\times\mathcal{P}} \left[\left(\max_{t_0} \Pi(t_0,z_0,p)\right)-T_1(z_0,p)\right]
\eeq
is increasing for $T$ smaller than some positive constant $T_L$. Moreover: 
\beq\label{lim:0}\lim_{T\rightarrow\,0^+} \max_{(z_0,p)\,\in\,\mathcal{Z}\times\mathcal{P}} \left[\left(\max_{t_0} \Pi(t_0,z_0,p)\right)-T_1(z_0,p)\right]=0\eeq
\end{theorem}

\begin{proof} 
From equation \eqref{end_comes}, we have:
\begin{eqnarray*}
\lefteqn{\max_{(z_0,p)\,\in\,\mathcal{Z}\times\mathcal{P}} \left[\left(\max_{t_0} \Pi(t_0,z_0,p)\right)-T_1(z_0,p)\right]}\\&&= \max_{(z_0,p)\,\in\,\mathcal{Z}\times\mathcal{P}}\left[
\frac{\mu}{\mu -\sigma}\left(\frac{1-e^{-mt_0^*(T,z_0,p)}}{1-e^{-mT}}\ T-t_0^*(T,z_0,p)\right)\right]\\
&~~~~~~~~~~~~~~~~~~~~~~~~~~~~~~&= \max_{p\,\in\,\mathcal{P}}\left[\frac{\mu}{\mu-\sigma}\max_{z_0\,\in\,\mathcal{Z}}
\left(\frac{1-e^{-mt_0^*(T,z_0,p)}}{1-e^{-mT}}\ T-t_0^*(T,z_0,p)
\right)
\right]
\end{eqnarray*}

We first concentrate on: 
\beq\label{worstz0}
\max_{z_0\,\in\,\mathcal{Z}}
\left(\frac{1-e^{-mt_0^*(T,z_0,p)}}{1-e^{-mT}}\ T-t_0^*(T,z_0,p)
\right)\eeq
for fixed $T$ and $p$. 
Let us introduce: \[T_l(\overline{z_0},\underline{z_0},p)=\min\left(\hat{T}(p),\frac{\overline{z_0}-\underline{z_0}}{2(\mu-\sigma)}\right)\]
with $\hat{T}(p)$ defined in equation \eqref{tm}. In the sequel, we consider that $T$ belongs to $\left(0,T_l(\overline{z_0},\underline{z_0},p)\right)$. On the one side, this ensures through Proposition \ref{prop:Tm}
that, independently of the actual parameters, the pest population (\ie $z$ in our new variables) always decreases after the invasion time $t_0$. 

On the other side, we have:
\[\frac{\overline{z_0}-\underline{z_0}}{T(\mu-\sigma)}>2
\]
Then, there exists $n\in\mathbb{N}$ and $z_{01},\ z_{02}\in[\underline{z_0},\overline{z_0}]$ such that:
\[\frac{z_{01}}{T(\mu-\sigma)}=n~~\text{and}~~\frac{z_{02}}{T(\mu-\sigma)}=n+1
\]
From Theorem \ref{best_period}, $T$ is thus an optimal period for both the initial conditions $z_{01}$ and $z_{02}$.

We first show by contradiction that for all $z_0\in(z_{01},z_{02})$, case (ii) of Lemma \ref{lem:worst_t0} holds for $k=n$. Indeed suppose that for a given $z_0\in(z_{01},z_{02})$, we have $k\leq (n-1)$, then:
\begin{eqnarray}
\max_{t_0}\Pi=(k+1)T-t_0^*
&\leq& (n-1+1)T-t_0^*\nonumber\\
&\leq& T_1(z_{01},p)-t_0^*\nonumber\\
&<& T_1(z_{0},p)-t_0^*
\label{eq:k<n}
\end{eqnarray}
since from \eqref{t1def}, $T_1$ is an increasing function of $z_0$. \eqref{eq:k<n} is not possible since from Theorem \ref{best_period}: $\dst\max_{t_0}\Pi\geq T_1(z_0,p)$ and $t_0^*\geq 0$. Therefore $k$ cannot be smaller or equal to $(n-1)$ for $z_0\in(z_{01},z_{02})$.

Now suppose that $k\geq(n+1)$ for some $z_0\in(z_{01},z_{02})$. From \eqref{eq:t0z0}, we have:
\begin{eqnarray}
z_0&=&(\mu-\sigma)kT+(\mu-\sigma)T+\sigma t_0^*-\mu T\frac{1-e^{-m t_0^*}}{1-e^{-mT}}\nonumber\\
&\geq& (\mu-\sigma)(n+1)T-\sigma(T- t_0^*)+\mu T\left(1-\frac{1-e^{-m t_0^*}}{1-e^{-mT}}\right)\nonumber\\
&\geq& z_{02} +\left[\frac{\mu(T-t_0^*)}{e^{mT}-1}\left( T\left(\frac{e^{m(T-t_0^*)}-1}{T-t_0^*}\right)-\frac{\sigma}{\mu}\left(e^{mT}-1\right)\right)\right]\label{eq:k>n}
\end{eqnarray}
The bracketted expression in \eqref{eq:k>n} is positive since:
\[\frac{e^{m(T-t_0^*)}-1}{T-t_0^*}\geq m
\]
and from \eqref{tm}:
\[\forall T\in (0,\hat{T}),~mT>\frac{\sigma}{\mu}(e^{mT}-1)\]
Hence, \eqref{eq:k>n} contradicts the hypothesis that $z_0<z_{02}$ so that $k$ cannot be greater or equal to $(n+1)$ for $z_0\in(z_{01},z_{02})$. 
Thus, we have shown that for all $z_0\in(z_{01},z_{02})$, case (ii) of Lemma \ref{lem:worst_t0} holds for $k=n$.

Since $T$ is an optimal period for the initial condition $z_{01}$, we should have by continuity:
\[\lim_{z_0\rightarrow z_{01}^+} \Pi(t_0^*(T,z_0,p))=(n+1)T-\lim_{z_0\rightarrow z_{01}^+}t_0^*(T,z_0,p)=T_1(z_{01},p)
\]
Then:
\begin{eqnarray*}
\lim_{z_0\rightarrow z_{01}^+}t_0^*(T,z_0,p)&=&(n+1)T-T_1(z_{01},p)\nonumber\\
&=& (n+1)T-\frac{z_{01}}{\mu-\sigma}T=T
\end{eqnarray*}
A similar argumentation on $z_{02}$ yields:
\begin{equation*}\lim_{z_0\rightarrow z_{02}^-}t_0^*(T,z_0,p)=0
\end{equation*}
From \eqref{eq:t0z0}, $t_0^*(T,z_0,p)$ is a continuous function of $z_0$. Then, as $z_0$ covers $[z_{01},z_{02}]$, $t_0^*(T,z_0,p)$ reaches its whole interval of definition $[0,T]$ so that we have:
\begin{eqnarray*}
 \max_{z_0\,\in\,[z_{01},z_{02}]}
\lefteqn{\left(\frac{1-e^{-mt_0^*(T,z_0,p)}}{1-e^{-mT}}\ T-t_0^*(T,z_0,p)
\right)}\\&~~~~~~~~~~~~~~~~~~~~~~~~~~~~~~~~~~~~~~&=\max_{t_0^*\in[0,T]}\left(\frac{1-e^{-mt_0^*}}{1-e^{-mT}}\ T-t_0^*
\right)\\
&&=\max_{z_0\,\in\,\mathcal{Z}}
\left(\frac{1-e^{-mt_0^*(T,z_0,p)}}{1-e^{-mT}}\ T-t_0^*(T,z_0,p)
\right)
\end{eqnarray*}
Differentiating $\left(\frac{1-e^{-mt_0^*}}{1-e^{-mT}}\ T-t_0^*\right)$ with respect to $t_0^*$, we show that it reaches its maximum for:
\begin{equation}\label{eq:hat_t0}
\hat{t_0^*}=\frac{1}{m}\ln\left(\frac{mT}{1-e^{-mT}}\right)
\end{equation}
$\hat{t_0^*}>0$ since $mT> 1-e^{-mT}$. In addition, $\hat{t_0^*}<T$ since    $ e^{mT}>\frac{mT}{1-e^{-mT}}$. Then:
\begin{eqnarray*}\max_{z_0\,\in\,\mathcal{Z}}
\left(\frac{1-e^{-mt_0^*(T,z_0,p)}}{1-e^{-mT}}\ T-t_0^*(T,z_0,p)
\right)&=& \frac{1-e^{-m\hat{t_0^*}}}{1-e^{-mT}}\ T-\hat{t_0^*}\\
&=& \frac{1}{m}\left[\frac{mT}{1-e^{-mT}}-1-\ln\left(\frac{mT}{1-e^{-mT}}
\right)
\right]\\&\triangleq& H(T,p)
\end{eqnarray*}
$H(T,p)$ is an increasing function of $T$ on $(0,T_l(\underline{z_0},\overline{z_0},p))$ since:
\[\frac{\partial H(T,p)}{\partial T} = \frac{e^{-mT}(e^{mT}-1-mT)}{(1-e^{-mT})^2}\left[1-\frac{1-e^{-mT}}{mT}\right]>0\]
Moreover, through the Hospital rule, we have for all $p\in\mathcal{P}$:
\[\lim_{T\rightarrow\ 0^+} H(T,p)=0\]

Let us introduce: \[T_L=\min_{p\in\mathcal{P}}T_l(\underline{z_0},\overline{z_0},p)>0\]
Through the above analysis, it has been shown that for $T\in(0,T_L)$ and $p\in\mathcal{P}$,  $H(T,p)$ is an increasing function of $T$ with $\lim_{T\rightarrow 0^+} H(T,p)=0$.

We now come back to the function defined in \eqref{rob:deviation}. We have:
\[\max_{(z_0,p)\,\in\,\mathcal{Z}\times\mathcal{P}} \left[\left(\max_{t_0} \Pi(t_0,z_0,p)\right)-T_1(z_0,p)\right]=\max_{p\in\mathcal{P}}\left(
\frac{\mu}{\mu-\sigma}\ H(T,p)\right)\]
Then, from $H(T,p)$ properties,  for $T$ smaller than $T_L$ the worst deviation of $\max_{t_0}\Pi$ from $T_1$ (according to $(z_0,p)\in\mathcal{Z}\times\mathcal{P}$) is an increasing function of $T$. Moreover:
\[\lim_{T\rightarrow\,0^+} \max_{(z_0,p)\,\in\,\mathcal{Z}\times\mathcal{P}} \left[\left(\max_{t_0} \Pi(t_0,z_0,p)\right)-T_1(z_0,p)\right]=0\]
which concludes the proof.\end{proof}

One can deduce the following Corollary from Theorem \ref{theo:rob}:

\begin{corollary}\label{col:rob}
Let $\dst\overline{T}=\min_{p\,\in\,{\mathcal P}}\hat{T}(p)$. Then:
\beq\label{robustness_crit}
\inf_{T\,\in\,(0,\overline{T})}\max_{(z_0,p)\,\in\,\mathcal{Z}\times\mathcal{P}}
\left[\left(\max_{t_0\,\in\,[0,T]}\Pi(t_0,z_0,p)\right)-T_1(z_0,p)\right]=0\eeq
is achieved for $T=0^+$.
\end{corollary}

\begin{proof}
From Theorem \ref{theo:rob}:
\beq\label{rob:fin1}
\inf_{T\,\in\,(0,T_L)}\max_{(z_0,p)\,\in\,\mathcal{Z}\times\mathcal{P}}
\left[\left(\max_{t_0\,\in\,[0,T]}\Pi(t_0,z_0,p)\right)-T_1(z_0,p)\right]=0
\eeq
and is achieved for $T=0^+$ since the argument of the infimum in \eqref{rob:fin1} is increasing of $T$ and tends to $0$ as $T$ does. 

To complete the proof of Corollary \ref{col:rob} we have to study:
\[\max_{(z_0,p)\,\in\,\mathcal{Z}\times\mathcal{P}}
\left[\left(\max_{t_0\,\in\,[0,T]}\Pi(t_0,z_0,p)\right)-T_1(z_0,p)\right]\]
for $T\in\left[T_L, \overline{T}\right)$. Notice that from $T_L$ definition, we have:
\[T_L=\min_{(z_0,p)\in\mathcal{Z}\times\mathcal{P}}\left(\min\left(\hat{T}(p),\frac{\overline{z_0}-\underline{z_0}}{2(\mu-\sigma)}\right)\right)=\min\left(\overline{T},\min_{(z_0,p)\in\mathcal{Z}\times\mathcal{P}} \left(\frac{\overline{z_0}-\underline{z_0}}{2(\mu-\sigma)}\right)\right)
\]
so that the set $\left[T_L, \overline{T}\right)$ is non-empty if and only if: $\overline{T}>
\min_{(z_0,p)\in\mathcal{Z}\times\mathcal{P}} \left(\frac{\overline{z_0}-\underline{z_0}}{2(\mu-\sigma)}\right)$. Suppose that this holds and that $T\in\left[T_L, \overline{T}\right)$, then we have:
\[\max_{z_0\in\mathcal{Z}}\left(\frac{1-e^{-mt_0^*(T,z_0,p)}}{1-e^{-mT}}-t_0^*(T,z_0,p)\right)>0\]
since otherwise it would be required that $\forall z_0\in\mathcal{Z}$ there would exist $n\in\mathbb{N}$ such that $T=\frac{z_0}{n(\mu-\sigma)}$ which is obviously not possible. Then, $\forall T\in[T_L,\overline{T})$:
\begin{eqnarray*}
\max_{(z_0,p)\,\in\,\mathcal{Z}\times\mathcal{P}}
\lefteqn{\left[\left(\max_{t_0\,\in\,[0,T]}\Pi(t_0,z_0,p)\right)-T_1(z_0,p)\right]=}\\&~~~~~~~~~~~~~~~~~~~~~~~~~&\dst\max_{p\in\mathcal{P}}\left[\frac{\mu}{\mu-\sigma}\max_{z_0\in\mathcal{Z}}
\left(\frac{1-e^{-mt_0^*(T,z_0,p)}}{1-e^{-mT}}-t_0^*(T,z_0,p)\right)\right]
\end{eqnarray*}
this latter being positive, which concludes the proof.\end{proof}

With the help of Theorem \ref{theo:rob}, we know that provided $T$ is not too large, the smaller $T$ is chosen, the less will be the worst case (according to parameters and initial condition uncertainty) deviation of $\max_{t_0}\Pi$ from its minimal value $T_1$. It has moreover been shown that as $T$ tends to $0$, the worst case  deviation of $\max_{t_0}\Pi$ from its minimal value goes to $0$ as well.

We conclude from this analysis that the robust choice of $T$, as defined by
the criterion (\ref{robustness_crit}) consists in choosing $T=0$. However, it is obviously not possible to achieve such a $T$ in practical applications: it would consist in continuously releasing predators into the crop culture. As an alternative, and since the argument of the infimum in (\ref{robustness_crit}) is an increasing function of $T$ for $T\in(0,T_L)$, it is then advised to take $T$ as small as possible to achieve the best possible robustness. 

A last remark should be made about the need for robustness. In the preceding
analysis, we have concentrated on a single model in the form:
\[
\dot z=S-y_p(t)
\]
We should remember that this model covers two cases: a local linearisation of the system (which yields a certain value of $S$) and a system that globally upper-bounds the actual nonlinear system (which yields another value of $S$). At the very least, the choice of $T$ that we made should work well in both those cases. We can then conclude that, even in the unrealistic case where the initial conditions and the parameters are very well-known, one should make a choice of $T$ that is sufficiently robust, so that both global
and local approximations are covered, \ie one should take $T$ as small as possible.

\section{Numerical Simulation}

To illustrate the robustness result stated in Theorem \ref{theo:rob} we performed a Monte Carlo like numerical simulation: an impulsive predator-prey model of the form \eqref{model} was considered with constant a priori chosen parameters for the continuous part of the model and $\mu$ verifying condition \eqref{cond_stab_glob}. A triplet of the remaining parameters was randomly chosen: first a release period  $T\in(0,T_L)$, the pest invasion moment $t_0\in[0,T)$ and the pest invasion level $x_0>\bar x$. The value of $T_L$ is easy to compute from \eqref{tm} since in this case the parameter set $\mathcal{P}$ is reduced to a singleton and that we have chosen a large  $\mathcal{Z}$ interval. The time to reach $\bar{x}$, $\Pi$, was then computed as well as its deviation from $T_1$. The process has been repeated until a set of $2.10^5$ different simulations was obtained and the results, the deviation of $\Pi$ from $T_1$, were plotted against the time period $T$.

The results of this numerical work is presented  on Figure \ref{fig:mc}, each of the dot corresponding to one of the $2.10^5$ simulations. What actually corresponds to Theorem \ref{theo:rob} is the upper envelope of the data set only (since it is the maximum with respect to parameters of the worst case, according to $t_0$, of $\Pi$).

\begin{figure}[h]\begin{center}
\includegraphics[width=.8\textwidth]{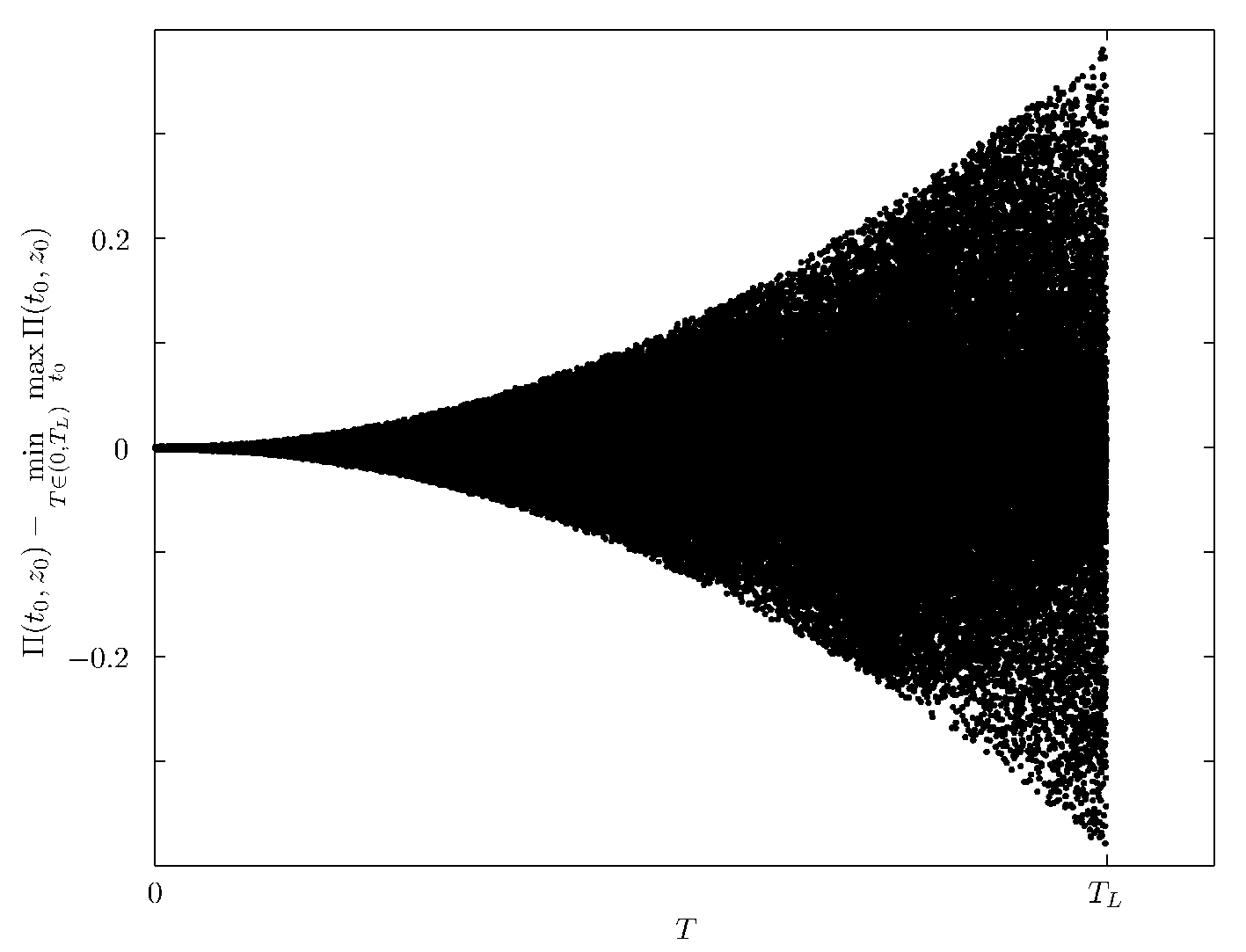}
\caption{Graphical illustration of Theorem \ref{theo:rob}. Each dot corresponds to the deviation of $\Pi$ from the minimal worst case $T_1$ for one of the $2.10^5$ randomly chosen triplet $T\in(0,T_L),~t_0\in(0,T)$ and $x_0>\bar{x}$.\label{fig:mc}}\end{center}
\end{figure}

As it can be noticed the smaller the choice of $T$, the smaller the positive deviation of $\Pi$ (thus of the crop damages) from its optimal value $T_1$. 
The price to pay is however that such a choice reduces also the negative deviation of $\Pi$ from $T_1$. Indeed an analysis similar to the one performed in Section \ref{sec:opt}, but focusing on: \[\min_{(z_0,p)\,\in\,\mathcal{Z}\times\mathcal{P}} \left[\left(\min_{t_0} \Pi(t_0,z_0,p)\right)-T_1(z_0,p)\right]\]
would show that this is a decreasing function of $T$ with limit $0$ as $T$ tends to $0^+$. Hence, if one feels lucky, a choice of $T$ small may not be the best one since one can obtain, fortunately, smaller $\Pi$ (thus smaller crop damages) if the actual $t_0$ is not the worst one.

\section{Conclusion}

In this contribution we studied a rather general predator-prey model with periodic impulsive additions  of predators (releases) that is relevant for biological control modelling. The impulsive releases of the biological control agents were modelled on the basis of a fixed budget to be spent per time period (release rate). As such it allowed to compare the different releases period as the same overall amount of predators was used. In a first step, it was shown that, provided the release rate of biocontrol agents was high enough, the biological control program allows to eradicate the prey (pest) population, whatever the value of the release period. Then the efficiency of different release periods (for a given release rate) were compared on the assumption that the biological control program is conducted on a preventative basis. It was shown that the time needed to eradicate an unforeseen pest invasion occurring at a worst moment had a minimum that could be achieved through the choice of a countable infinite number of release period. However, these release period values were strongly dependent on the model parameters and a robustness analysis of this optimisation problem showed that the smaller the release period, the less the time needed to eradicate the pest invasion, for the worst invasion moment and accounting for parameter uncertainty. Hence, as an advice to biological control practitioners, our study indicates that it is less risky to release frequent small amounts of biocontrol agents rather than massive infrequent ones. This result (frequent releases of small doses yield better control) is similar to the ones of \cite{Yano1989,Yano1989a,Fenton2001} that were all obtained from simulatory models. Such a biocontrol strategy is also advocated by \cite{Williams2001} from real life experiments and by \cite{Skirvin2002} on the basis of a simulatory model. However, to our knowledge, this contribution is the first to recommend such a strategy on the basis of a mathematical analysis.

In a broader context, we believe that the techniques used in this work may also be of interest for optimisation problems on impulsive models from other applied fields: for instance, optimisation of vaccination strategies in epidemiology \cite{Shulgin1998,D'onofrio2002a,Hui2004a} or optimisation of chemotherapy in cancer treatments \cite{Lakmeche2000,D'onofrio2004,D'onofrio2005}.

\end{document}